\documentclass[11pt]{amsart}

\usepackage[utf8x]{inputenc} 
\usepackage[a4paper]{geometry} 
\usepackage{enumerate}  
\usepackage{amsmath} 
\usepackage{amssymb}
\usepackage{color,colortbl}
\usepackage{multirow}
\usepackage{hyperref}  
\usepackage{verbatim}
                 \hypersetup{ pdfborder={0 0 0}, 
                              colorlinks=true, 
                              linkcolor=blue, 
                              urlcolor=blue,
                              citecolor=blue,
                              linktoc=page,
                              pdfauthor={Giovanni Staglian{\`o}}, 
                              pdftitle={Special cubic Cremona transformations of P6 and P7}} 
\definecolor{Gray}{gray}{0.9}                            
\usepackage{mathptmx}                      
   
\theoremstyle{plain} 
\newtheorem{proposition}{Proposition}[section] 
\newtheorem{theorem}[proposition]{Theorem} 
\newtheorem{lemma}[proposition]{Lemma} 
 
\newtheorem{fact}[proposition]{Fact} 
\theoremstyle{definition}

\theoremstyle{remark} 
\newtheorem{remark}[proposition]{Remark} 
                              
\newcommand{\I}{{\mathcal{I}}}     
\renewcommand{\O}{{\mathcal{O}}}   
\newcommand{\T}{{\mathcal{T}}}    
    
\newcommand{\B}{{\mathfrak{B}}} 
\newcommand{\R}{{\mathfrak{R}}} 
\newcommand{\N}{{\mathcal{N}}}     
 
\newcommand{\ZZ}{{\mathbb{Z}}} 
      
\newcommand{\PP}{{\mathbb{P}}}          

\providecommand{\Sec}{\mathop{\rm Sec}\nolimits} 
 
\providecommand{\codim}{\mathop{\rm codim}} 

\numberwithin{equation}{section}

\title[Special cubic Cremona transformations of $\mathbb{P}^6$ and $\mathbb{P}^7$]{Special cubic Cremona transformations of $\mathbb{P}^6$ and $\mathbb{P}^7$} 
\author[G. Staglian\`o]{Giovanni Staglian\`o} 
\address{Instituto de Matem\'atica -- Universidade Federal Fluminense} 
\date{\today} 
\email{\href{mailto:giovannistagliano@gmail.com}{giovannistagliano@gmail.com}} 
\thanks{This work was supported by a BJT fellowship from CAPES (N. A028/2013).} 

\begin{document}

\begin{abstract} A famous result of B. Crauder and S. Katz (1989) concerns the classification of special Cremona transformations whose base locus has dimension at most two. Furthermore, they also proved that a special Cremona transformation with base locus of dimension three has to be one of the following: 1) a quinto-quintic transformation of $\mathbb{P}^5$; 2) a cubo-quintic transformation of $\mathbb{P}^6$; or 3) a quadro-quintic transformation of $\mathbb{P}^8$. Special Cremona transformations as in case 1) have been classified by L. Ein and N. Shepherd-Barron (1989), while in our previous work (2013), we classified special quadro-quintic Cremona transformations of $\mathbb{P}^8$. The main aim here is to consider the problem of classifying special cubo-quintic Cremona transformations of $\mathbb{P}^6$, concluding the classification of special Cremona transformations whose base locus has dimension three.
\end{abstract}

\maketitle

\section*{Introduction} 
In this paper, 
we investigate special Cremona transformations, i.e. birational 
transformations $\varphi:\PP^n\dashrightarrow\PP^n$ 
of a complex projective space into itself,
whose base locus $\B\subset\PP^n$ is a smooth and irreducible variety.
One says that $\varphi$ is of type $(\delta_1,\delta_2)$ if the 
linear system defining it (resp. defining its inverse) 
belongs to $|\O_{\PP^n}(\delta_1)|$ (resp. $|\O_{\PP^n}(\delta_2)|$).
When the dimension of the base locus is at most two, B. Crauder and S. Katz obtained the following results:
\begin{theorem}[\cite{crauder-katz-1989}, see also \cite{katz-cubo-cubic}]\label{classificazioneCasoCurve}
 Let $\varphi$ be a special 
 Cremona transformation of $\PP^n$ whose base locus $\B$ has dimension $1$. Then 
 one of the two following cases occurs:
 \begin{enumerate}[(1)]
  \item\label{cubocubicinP4} $n=3$, $\varphi$ is of type $(3,3)$, and $\B$ is a curve of degree $6$ and genus $3$;
  \item $n=4$, $\varphi$ is of type $(2,3)$, and $\B$ is a curve of degree $5$ and genus $1$.
 \end{enumerate}
\end{theorem}
In the same paper,  they also proved the following: 
\begin{theorem}[\cite{crauder-katz-1989}, see also \cite{semple-tyrrell,cite2-semple}]
 Let $\varphi$ be a special 
 Cremona transformation of $\PP^n$ whose base locus $\B$ has dimension $2$. Then 
 one of the following cases occurs:
 \begin{enumerate}[(1)]
  \item $n=4$, $\varphi$ is of type $(3,2)$, 
  and $\B$ is a quintic elliptic scroll;
  \item $n=4$, $\varphi$ is of type $(4,4)$, 
  and $\B$ is a degree $10$ determinantal surface  given 
  by the $4\times4$ minors of a $4\times5$ matrix of linear forms; 
  \item $n=5$, $\varphi$ is of type $(2,2)$, and $\B$ is the Veronese surface;
  \item $n=6$, $\varphi$ is of type $(2,4)$, and $\B$ is a septic elliptic scroll; 
  \item $n=6$, $\varphi$ is of type $(2,4)$, and $\B$ is $\PP^2$ blown up at 
  eight points and embedded in $\PP^6$ as an octic surface by 
  quartic curves passing  
  through all eight points.
 \end{enumerate}
\end{theorem}
Furthermore, when $\B$ has dimension $3$, 
we have the following possible cases (see \cite[Corollary~1]{crauder-katz-1991},
and \cite[Theorem~3.2]{ein-shepherdbarron} for the second part of the statement
in \eqref{casecrei} below):
\begin{proposition}\label{threefoldsCre}
 Let $\varphi$ be a special 
 Cremona transformation of $\PP^n$ whose base locus $\B$ has dimension $3$. Then 
 one of the following cases occurs:
 \begin{enumerate}[(1)]
  \item\label{casecrei} $n=5$, 
  $\varphi$ is of type $(5,5)$, 
 and $\B$ 
 is a degree $15$ determinantal threefold
 given by the 
  $5\times 5$ minors of a $5\times 6$ matrix of linear forms; 
  \item\label{cubicCase} $n=6$ and $\varphi$ is of type $(3,5)$;
  \item\label{quadraticCase} $n=8$ and $\varphi$ is of type $(2,5)$. 
 \end{enumerate}
\end{proposition}
In \cite{note2}, we classified special  Cremona transformations as in 
case \eqref{quadraticCase} above, by showing the following:
\begin{theorem}[\cite{note2}]\label{thmQuadratic}
  Let $\varphi$ be  a special 
 quadratic Cremona transformation of $\PP^n$  whose base locus $\B$ has dimension $3$.
 Then $n=8$, $\varphi$ is of type $(2,5)$, and one of the two following cases occurs:
 \begin{enumerate}[(1)]
  \item\label{case25a} $\B$ is a scroll 
  over a rational surface $Y$ with $K_{Y}^2=5$,
  and it is embedded in $\PP^8$ as a threefold of 
  degree $12$ and sectional genus $6$;
  \item\label{case25b} $\B$ is  the blow-up of a Fano variety at one point, 
  embedded in $\PP^8$ as a threefold of degree $13$ and sectional genus $8$.
 \end{enumerate}
\end{theorem}
In \cite{note3}, 
we  exhibited an example of Cremona transformation as in case \eqref{case25a} above, while  examples 
 of transformations as in  case \eqref{case25b} 
had been constructed in \cite{hulek-katz-schreyer}.
The main aim of this paper is to prove Theorem \ref{MainTheorem} below,
which describes the transformations as in case \eqref{cubicCase} 
 of Proposition \ref{threefoldsCre}, concluding the classification of special Cremona
 transformations whose base locus has dimension three.
\begin{theorem}\label{MainTheorem}
 Let $\varphi$  be a special 
 cubic Cremona transformation of $\PP^n$ whose base locus $\B$ has dimension $3$.
 Then $n=6$, $\varphi$ is of type $(3,5)$, and one of the two following cases occurs:
\begin{enumerate}[(A)]
 \item\label{LogGeneralCaseThm} $\B$ is a threefold of degree $14$, sectional genus $15$, 
 and it is Pfaffian (i.e. given by the Pfaffians of a skew-symmetric matrix) 
 with trivial canonical bundle;
 \item\label{FibrationCaseThm} $\B$ is a conic bundle over $\PP^2$, embedded in $\PP^6$
 as a threefold of degree $13$ and sectional genus~$12$.
\end{enumerate}
 \end{theorem}
 An example as in the first case of 
 Theorem \ref{MainTheorem} 
 can be obtained by taking 
 the Pfaffians of the principal $6\times 6$ minors 
 of a general $7\times 7$ skew-symmetric matrix of linear forms on $\PP^6$, while 
 an example as in the second case is a bit more complicated; see Section \ref{sez examples}.

From another point of view, if $\varphi$ 
is a special Cremona transformation of $\PP^n$ and $n\leq6$, then
the dimension of its base locus is at most three 
(see \cite[Theorem~3.2]{ein-shepherdbarron}). Thus  
Theorem \ref{MainTheorem} also 
concludes the classification of the special Cremona transformations 
of $\PP^n$ with $n\leq6$. 
We know no examples of special Cremona transformations of $\PP^7$, 
but the second main result of this paper is the following:
  \begin{theorem}\label{CremonaP7}
 Let $\varphi$ be a special Cremona transformation of $\PP^7$.
Then  $\varphi$ is of type $(3,3)$  and  
its base locus is a fourfold of degree $12$ and sectional genus $10$, which
is a fibration over $\PP^1$ whose generic fibre is a sextic del Pezzo threefold.
\end{theorem}

For the convenience of the reader, in Table \ref{table: classCremona} below 
we summarize the classification of the special Cremona transformations 
of $\PP^n$ with either $n\leq7$ or dimension 
of the base locus $r\leq3$. 
 \begin{table}[htbp]
\centering
\begin{tabular}{|c|c|c|c|l|}
\hline
  & $r$ & $n$ & Projective degrees &   Description of the base locus   \\  
\hline 
I & \multirow{2}{*}{$1$} & $3$  & $3, 3$ &  determinantal curve   \\
II & & $4$  & $2, 4, 3$ &  quintic elliptic curve   \\
\hline 
III & \multirow{5}{*}{$2$} & $4$  & $3, 4, 2$ &  quintic elliptic scroll in lines   \\
IV & & $4$  & $4, 6, 4$ &   determinantal surface   \\
V & & $5$  & $2, 4, 4, 2$ &   Veronese surface   \\
VI & & $6$  & $2, 4, 8, 9, 4$ &  septic elliptic scroll in lines   \\
VII & &$6$  & $2, 4, 8, 8, 4$ &   rational octic surface   \\
\hline 
VIII & \multirow{5}{*}{$3$} &$5$  & $5, 10, 10, 5$ &   determinantal threefold  \\
IX & &$8$  & $2, 4, 8, 16, 20, 14, 5$ &   scroll over rational surface with $K^2=5$   \\
X  & &$8$  & $2, 4, 8, 16, 19, 13, 5$ &  Fano threefold with one point blown up  \\
XI  & &$6$  & $3, 9, 13, 11, 5$ &   Pfaffian threefold   \\ 
 XII  & &$6$  &  $3, 9, 14, 12, 5$ &   conic bundle over a plane   \\
\hline
\rowcolor{Gray}
 XIII  & \multirow{1}{*}{$4$}  & $7$  &  $3$, $9$, $15$, $15$, $9$, $3$ &  del Pezzo fibration over a line \\
\hline
\end{tabular}
 \caption{Special Cremona transformations of $\PP^n$ with either $n\leq7$
 or $r\leq3$.} 
\label{table: classCremona} 
\end{table}

 The outline of the paper is as follows. In Section \ref{prelim sec},
 we give some background information; in particular, we recall 
  some  adjunction-theoretic results.
  
In Section \ref{section: cubic}, we prove Theorem \ref{MainTheorem},
and 
we will proceed
essentially in four steps. Firstly, we 
compute the numerical invariants of $\varphi$ and $\B$ as functions 
of $\lambda=\deg(\B)$ and of the sectional genus $g$ of $\B$. In particular,
we determine the Hilbert polynomial of $\B$ (Lemma \ref{hilbertPolB}), 
the degrees of the Chern and Segre 
classes of the tangent and normal bundle of $\B$ (Lemma \ref{lemmaChern}), 
and the projective degrees of $\varphi$ (Lemma \ref{CremonaDegrees}).
Secondly, 
applying general inequalities for the invariants of threefolds and 
for the projective degrees of Cremona transformations,
we reduce drastically the number of pairs $(\lambda,g)$
for which such a transformation may exist
(see Lemma \ref{lemma extreme}).
Thirdly, we apply adjunction theory to deduce a first rough result 
(Proposition~\ref{casiPossi}):  either $(\B,H_{\B})$ is minimal of log-general type, 
or it admits a reduction to a conic bundle 
 over a surface.
Four and last, we refine our result by applying to the general 
 hyperplane section  of $\B$ a formula, due to P. Le Barz, 
concerning the number of $4$-secants of surfaces in $\PP^5$.

In Section \ref{sec cubo-cubic},
we prove Theorem \ref{CremonaP7},
basically by altering slightly 
the proof of Theorem \ref{MainTheorem} 
 in order to study 
 the restriction of a special Cremona transformation 
 of $\PP^7$ to a general hyperplane $\PP^6\subset\PP^7$;
 indeed one has that this restriction is  a cubic transformation 
 of $\PP^6$ (into a hypersurface of $\PP^7$)
  with a smooth irreducible threefold as base locus.
 
In Section \ref{sez examples}, among others examples,
we explain how to construct an example of Cremona transformation 
 as in case~(\ref{FibrationCaseThm}) of Theorem \ref{MainTheorem}.   
For this, we will use {\sc Macaulay2} \cite{macaulay2} with the \href{https://github.com/Macaulay2/M2/blob/master/M2/Macaulay2/packages/Cremona.m2}{\emph{Cremona}} package \cite{packageCremona}.
 
We point out that some proofs in this paper  
are reduced to purely mechanical procedures of 
finding out the elements of a given
finite set of tuples of integers that satisfy certain given relations. 
So a computer is very helpful, although not indispensable.

\section{Preliminaries}\label{prelim sec}
\subsection{Overview on adjunction theory for threefolds}\label{adjunctionThery}
Since some adjunction theoretic results are central for us, 
in this subsection
we summarize them for the convenience of the reader.
For details, proofs, and history on these results, we refer to the papers 
\cite{fujita-polarizedvarieties,beltrametti-sommese,ionescu-smallinvariants,ionescu-adjunction,sommese-adjunction-theoretic,sommese-adjunction-mapping,beltrametti-biancofiore-sommese-loggeneral} 
       and references therein.
 As usual, we do not distinguish between 
line bundles and divisors. 

Let $X$ be a smooth irreducible threefold,
$H$ a very ample divisor on $X$, and denote by
$K$ the canonical divisor of $X$.
The degree of the polarized pair $(X,H)$ is  $H^3$, and the sectional genus $g=g(X,H)$ 
is defined by $2g-2=(K+2H)H^2$. 

A pair $(X',H')$, where $X'$ is a smooth irreducible threefold and 
$H'$ an ample divisor on $X'$, is said to be a \emph{reduction} of $(X,H)$ if 
there is a morphism $\pi:X\rightarrow X'$ expressing $X$ 
as the blowing-up of $X'$ at a finite 
number of distinct points $p_1,\ldots,p_{\nu}$ (not 
infinitely near),
and moreover, denoting by  $E_i=\pi^{-1}(p_i)$ the exceptional divisors 
and by $K'$ the canonical divisor of $X'$, we have 
 $H\approx\pi^{\ast}H'-\sum_{i=1}^{\nu} E_i$ 
or equivalently $K+2H\approx\pi^{\ast}(K'+2H')$.
Of course, we have $H'^3=H^3+\nu$ and $g(X',H')=g(X,H)$. 
When $K+2H$ is nef and big,
there is a reduction $(X',H')$, uniquely determined by $(X,H)$, 
such that  $K'+2H'$ is ample. So, in this case, we refer to $(X',H')$ 
as \emph{the (minimal) reduction} of $(X,H)$, and 
denote by $\nu=\nu(X,H)$ the number of points blown up on $X'$.

Consider the complete linear system $|K+2H|$ on $X$. It is base-point free 
unless $(X,H)$ is one of the following:
 $(\PP^3,\O_{\PP^3}(1))$, 
the quadric $(Q^3,\O_{Q^3}(1))$, or a scroll over a curve. 
The map $\Phi=\Phi_{|K+2H|}$, when defined, is called \emph{the adjunction map}, and 
we write $\Phi=\sigma\circ \hat{\Phi}$ for the Remmert-Stein factorization of $\Phi$;
so $\hat{\Phi}$ is a morphism with connected fibers 
onto a normal variety, and 
$\sigma$ is a finite map.
\begin{fact}\label{listaCasi}
 If the adjunction map $\Phi$ is defined, then there are the following possibilities:
 \begin{enumerate}[(I)]
  \item $\dim\Phi(X)=0$ and 
  $(X,H)$ is a del Pezzo variety, i.e. a Fano variety of coindex $2$;
  \item $\dim\Phi(X)=1$ and $\hat{\Phi}$ expresses $X$ as a
  quadric fibration over a curve (i.e. 
  any closed fibre of $\hat{\Phi}$ is embedded as a quadric surface by the restriction of $H$);
  \item $\dim\Phi(X)=2$ and $\hat{\Phi}$ expresses $X$ as a scroll over a smooth surface 
  (i.e. any fibre of $\hat{\Phi}$ is embedded as a line by the restriction of $H$);
  \item $\dim\Phi(X)=3$, $X':=\hat{\Phi}(X)$ is smooth, $H':=\hat{\Phi}_{\ast}(H)$ is ample, and $(X',H')$ 
  is the minimal reduction of $(X,H)$; moreover $K'+2H'$ 
  is very ample ($K'$ stands for the canonical divisor of $X'$).
 \end{enumerate}
\end{fact}
Thus, except for a small list of well understood cases, $K+2H$ is nef and big, 
and $(X,H)$ admits a reduction $(X',H')$ with $K'+2H'$ ample. 
When moreover $K'+H'$ is nef and big, 
one says that $(X',H')$ is of \emph{log-general type}.
In this case, a smooth member of $|H'|$ is a minimal surface of general type. 
We have the following:
\begin{fact}\label{NefBigFact}
 Let $(X',H')$ be a reduction of $(X,H)$. Then $K'+H'$ is nef and big unless either:
 \begin{enumerate}[(I)]
  \item\label{P3red} $(X',H')\simeq(\PP^3,\O_{\PP^3}(3))$;
  \item\label{Q3red} $(X',H')\simeq(Q,\O_Q(2))$, where $Q$ is a quadric in $\PP^4$;
  \item\label{VeroneseFibration} $(X',H')$ is a Veronese fibration over a smooth curve $Y$,
  hence there is a surjective  morphism $\pi:X'\rightarrow Y$ 
  whose general fibre is $(\PP^2,\O_{\PP^2}(2))$ and 
  $2K'+3H'\approx\pi^{\ast}\mathcal{L}$ for some ample line bundle $\mathcal{L}$ on $Y$;
  \item\label{Mukaired} $(X',H')$ is a Mukai variety, i.e. a Fano variety of coindex $3$;
  \item\label{delPezzoFibration} $(X',H')$ is a del Pezzo fibration over a smooth curve $Y$,
  hence there exists a surjective morphism $\pi:X'\rightarrow Y$  and 
  $K'+H'\approx\pi^{\ast}\mathcal{L}$ for some ample line bundle $\mathcal{L}$ on $Y$;
  \item\label{quadricFibrOverSurface} $(X',H')$ is a conic bundle over a surface $Y$, 
  hence there exists a surjective morphism $\pi:X'\rightarrow Y$ 
  and $K'+H'\approx\pi^{\ast}\mathcal{L}$ for some ample line bundle $\mathcal{L}$ on $Y$.
 \end{enumerate}
\end{fact}
The \emph{pluridegrees} of $(X',H')$, for $j=0,\ldots,3$, where
$(X',H')$ is the reduction of $(X,H)$,
are  defined as
$d_j=d_j(X',H')={(K'+H')}^j\,{H'}^{3-j}$. 
The following result 
collects 
several inequalities, 
proved in \cite{beltrametti-biancofiore-sommese-loggeneral},
for the pluridegrees in the case when $(X',H')$ is of log-general type.
\begin{fact}\label{LogGeneralIneq}
Assume that $K+2H$ is nef and big, 
and let $(X',H')$ be the reduction of $(X,H)$.
If $K'+H'$ is nef and big, then
 the following inequalities hold:
 \begin{enumerate}[(i)]
  \item\label{trivialIneq} $d_1\geq1$,  $d_2\geq3$,  $d_3\geq1$;
  \item\label{hodgeIneq} $d_1^2\geq d_{2} d_{0}$, $d_2^2\geq d_{3} d_{1}$;
  \item\label{hodge2Ineq} $d_1^3\geq d_{3} d_{0}^2$, $d_2^3\geq d_{3}^2 d_{0}$;
  \item\label{lemma1p3} $5d_1\geq d_0$, $5d_2\geq d_1$, $5d_3\geq d_2$; 
  \item\label{prop2p8}  if it is not true that $d_3=1$, $d_2=5$, and $d_1\leq25$,
                        then  $4d_1\geq d_0$, $4d_2\geq d_1$, and $4d_3\geq d_2$; 
  \item\label{MiyaokaYauNoether} $2\left(\chi(\O_{X'}) - \chi(\O_{X'}(-H'))\right)-6\leq d_2<9\left(\chi(\O_{X'}) - \chi(\O_{X'}(-H'))\right)$;
  \item\label{prop2p11p1} $\left(3d_2+2d_1-d_0+12(\chi(\O_{X'}) - \chi(\O_{X'}(-H')))\right)/32\geq\chi(\O_{X'})$; 
  \item\label{prop2p11p2} $2d_3+7d_2+12d_1-3d_0+30\chi(\O_{X'}) + 18\chi(\O_{X'}(-H'))\geq0$;
  \item\label{prop2p11p3} $24\left(\chi(\O_{X'}) - \chi(\O_{X'}(-H'))\right)+2g-2-2d_2-c_3(\T_{X'})\geq 0$.
 \end{enumerate}
\end{fact}
\subsection{Projective degrees of Cremona transformations} 
Recall that the $k$-th projective degree $\deg_k(\varphi)$ 
of a Cremona transformation $\varphi:\PP^n\dashrightarrow\PP^n$ 
is the degree of 
the inverse image $\overline{\varphi^{-1}(\PP^{n-k})}$  of a general linear subspace 
$\PP^{n-k}\subseteq\PP^n$;  
in other words, the projective degrees of $\varphi$ are the multidegree of its graph,
considered as a cycle on $\PP^n\times\PP^n$. 
The following general 
result puts 
restrictions on the possible projective degrees of  Cremona transformations. 
The first part is easy and due to L. Cremona; 
the second part is an application of Hodge type inequalities
contained in \cite[Corollary~1.6.3]{lazarfeld-positivity}; see also 
\cite[Subsection~1.4]{dolgachev-cremonas}.  
\begin{fact}\label{factCremona}
 Let $\varphi:\PP^n\dashrightarrow\PP^n$ be a Cremona transformation. 
 The following hold:
 \begin{itemize}
  \item $1\leq\deg_{i+j}(\varphi)\leq  \deg_{i}(\varphi)\,\deg_{j}(\varphi)$, for $0\leq i,j\leq n$, with $i+j\leq n$;
  \item $\deg_{i-1}(\varphi)\,\deg_{i+1}(\varphi)\leq {\deg_i(\varphi)}^2$, for $1\leq i\leq n-1$. 
 \end{itemize}
\end{fact}
The next result concerns the computation of the projective degrees.
The only non-trivial part is the second equality of (\ref{segresDegs}), 
and this is essentially shown in \cite[p.~291]{crauder-katz-1989}.
It is
 a special case of 
the well-known relationship 
between the projective degrees of a rational map  from $\PP^n$ 
and the  
Segre class of the base locus in $\PP^n$ 
(see \cite[Proposition~4.4]{fulton-intersection}  
and \cite[Subsection~2.3]{dolgachev-cremonas}). 
\begin{fact}\label{factSpecialCremona2}
Let $\varphi:\PP^n\dashrightarrow\PP^n$ be a special Cremona transformation 
    of type $(\delta_1,\delta_2)$. Let $\B$ be its base locus, $r=\dim (\B)$,  and 
    let $s_i(\B)=s_i(\N_{\B,\PP^n})\,  H_{\B}^{r-i}$ be the degree of 
  the $i$-th Segre class of the normal bundle of $\B$. 
  Then the following hold for $k=0,\ldots,n$:  
\begin{gather}
\deg_{0}(\varphi)=\deg_{0}(\varphi^{-1})=1,\quad \deg_{1}(\varphi)=\delta_1 ,\quad  \deg_{1}(\varphi^{-1})=\delta_2; \\
\begin{split}
\label{segresDegs}
     \deg_{k}(\varphi^{-1}) & =\deg_{n-k}(\varphi) \\ &= \delta_1^{n-k}-  \binom{n-k}{r-k}\,\delta_1^{r-k}\,\deg(\B)  -\sum_{i=k}^{r-1}\binom{n-k}{i-k}\,\delta_1^{i-k}\,s_{r-i}(\B)  .
\end{split}
     \end{gather}
\end{fact}
\subsection{Some general constraints on special Cremona transformations}
Let $\varphi$ be a special Cremona transformation of type 
$(\delta_1,\delta_2)$ of $\PP^n$, and let $\B$ (resp. $\B'$) denote 
the base locus of $\varphi$ (resp. of $\varphi^{-1}$).
In \cite[Lemma~2.4]{ein-shepherdbarron} (see also \cite{crauder-katz-1989}) 
the following formulas are obtained:
\begin{equation}\label{dimensionBaseLocus}
n
=\frac{(\delta_1\delta_2-1)\,\dim\B+(\delta_1+1)\delta_2-2}{(\delta_1-1)\delta_2}
= \frac{(\delta_1\delta_2-1)\,\dim\B'+(\delta_2+1)\delta_1-2}{(\delta_2-1)\delta_1}  ,
\end{equation}
and one has $2\leq \delta_1,\delta_2\leq n$ and $1\leq\dim\B,\dim\B'\leq n-2$.
When $n\geq6$, one has $\delta_1\leq n-1$ and $\dim\B\leq n-3$, 
see \cite[Theorem~3.2]{ein-shepherdbarron}.

Let $k_0=k_0(\varphi)=(n-\dim\B)\delta_1-(n+1)$.
A simple application of the Kawamata-Viehweg vanishing theorem applied on 
 the blowing-up of $\PP^n$ along the base locus of $\varphi$ gives 
 the following
 (see e.g.  \cite[Lemma~4.3]{note}, \cite[Proposition~4.2(ii)]{russo-qel1}, or \cite[Corollary~2.12]{vermeire}):
 \begin{equation}\label{vanishingLemma}
  h^i(\PP^n,\I_{\B,\PP^n}(k))=0,\quad \mbox{for } i\geq1 \mbox{ provided } k\geq k_0 .
 \end{equation}
As one has $
            h^0(\PP^n,\I_{\B,\PP^n}(\delta_1 -1))=0$ 
and $h^0(\PP^n,\I_{\B,\PP^n}(\delta_1))= n+1 $, 
by \eqref{vanishingLemma} one always gets at least two conditions for the Hilbert polynomial 
of $\B$.
Note also that for $k\geq k_0+1$, \eqref{vanishingLemma}
is a consequence of the main result in \cite{bertram-ein-lazarsfeld}. 
\subsection{Four-secants of surfaces in \texorpdfstring{$\PP^5$}{P5}}
Recall the formula, due to P. Le Barz, calculating the number  of $4$-secant lines 
of a surface $S\subset\PP^5$.
\begin{fact}[\cite{barz-quadrisecantes}, see also \cite{barz-multisecantes}]\label{multisecants}
Let $S\subset\PP^5$ be a smooth surface that contains at most 
finitely many lines
and put
 $\kappa=K_{S}\,H_{S}$, $\zeta=K_{S}^2$, $\theta=c_2(\T_{S})$, and $\lambda=H_{S}^2$. Then 
the number of $4$-secant lines of $S$, if finite, is given by 
\begin{multline*}
24^{-1}\,\left(3 {\lambda}^4-
36 {\kappa} {\lambda}^2-
6 {\zeta} {\lambda}^2+
6 {\theta} {\lambda}^2-
90 {\lambda}^3+
78 {\kappa}^2   \right. \\ \left.  +
30 {\kappa} {\zeta}+
3 {\zeta}^2 -
30 {\kappa} {\theta}-
6 {\zeta} {\theta} +
3 {\theta}^2  + 
612 {\kappa} {\lambda}+
116 {\zeta} {\lambda}-
100 {\theta} {\lambda}    \right. \\ \left. +
855 {\lambda}^2-
1980 {\kappa}-
510 {\zeta}+
294 {\theta}-
2466 {\lambda}\right)
-\sum_{\begin{subarray}{l} l\, \mathrm{line}\\l\subset S\end{subarray}}\binom{5+l^2}{4} .
\end{multline*}
\end{fact}
\section{Proof of Theorem \ref{MainTheorem}}\label{section: cubic}
Let $\varphi$ be a special cubic 
 Cremona transformation  whose base locus has dimension $3$.
By \eqref{dimensionBaseLocus}
 this is the same as saying that
$\varphi$ is a special cubic Cremona transformation of $\PP^6$.
 Then $\varphi$ is of type $(3,5)$, and  
the base locus $\B\subset\PP^6$ is 
 a smooth irreducible threefold 
 cut out by $7$ cubic hypersurfaces.
\subsection{Numerical  invariants of the transformation}
In this subsection, we compute the numerical invariants of $\B$ and $\varphi$ as 
functions of $\lambda$ and $g$, where $\lambda$ and $g$ denote, respectively, the degree and 
the sectional genus of $\B$.
\begin{lemma}\label{lemma: cohomology ideal} The following hold:
 \begin{itemize} 
  \item $h^0(\PP^6,\I_{\B,\PP^6}(1))=h^0(\PP^6,\I_{\B,\PP^6}(2))= 0$ and $h^0(\PP^6,\I_{\B,\PP^6}(3))=7$;
  \item $h^i(\PP^6,\I_{\B,\PP^6}(j))=0$, for every $i\geq 1$ and $j\geq2$;
  \item $h^1(\PP^6,\I_{\B,\PP^6}(1))=0$, i.e. $\B$ is linearly normal.
 \end{itemize}
\end{lemma}
\begin{proof}
 The first statement is contained in \cite[Proposition~2.5]{ein-shepherdbarron},
 and the second follows by \eqref{vanishingLemma}.
  If $\B\subset\PP^6$ is not linearly normal, then 
 there exists a smooth nondegenerate threefold $X\subset\PP^7$ with $\Sec(X)\neq\PP^7$ 
and 
 such that $X$ is projected isomorphically onto $\B$.
 But such $X$'s are completely classified in \cite{fujita-3-fold}, and 
 none of them may be possible  in our case. 
\end{proof}
\begin{lemma}\label{hilbertPolB}
The Hilbert polynomial of $\B$ is 
\begin{multline*}
\chi(\O_{\B}(tH_{\B}))
= \lambda \,\binom{t+3}{3}+(-\lambda-g+1)\,\binom{t+2}{2}   \\
 +(-6\lambda+4g+45)\,(t+1)+(14\lambda-6g-113) .
\end{multline*}
\end{lemma}
\begin{proof}
 From Lemma \ref{lemma: cohomology ideal} and 
the exact sequence
$0\rightarrow \I_{\B,\PP^6}\rightarrow \O_{\PP^6}\rightarrow \O_{\B}\rightarrow 0$,
we deduce two conditions for the Hilbert polynomial of $\B$. That is
\begin{align*}
\chi(\O_{\B}(3)) &= h^0(\O_{\B}(3)) = h^0(\O_{\PP^6}(3)) - h^0(\I_{\B,\PP^6}(3)) = 77, \\
\chi(\O_{\B}(2)) &= h^0(\O_{\B}(2)) = h^0(\O_{\PP^6}(2)) = 28.
\qedhere
\end{align*}
\end{proof}
 \begin{lemma}\label{lemmaChern}  
 Let $K_{\B}$ and $H_{\B}$ be, respectively, a canonical divisor and
 a hyperplane section divisor of $\B$.
 For $i=1,2,3$, 
 let $c_i=c_i(\T_{\B})\,  H_{\B}^{3-i}$ (resp. $\mathfrak{c}_i=c_i(\N_{\B,\PP^6})\,  H_{\B}^{3-i}$) be
 the degree of the $i$-th Chern class of the tangent (resp. normal) bundle of $\B$,
 and let 
 $s_i=s_i(\N_{\B,\PP^6})\,  H_{\B}^{3-i}$ (resp. $\mathfrak{s}_i=s_i(\T_{\B})\,  H_{\B}^{3-i}$) be
 the degree of the $i$-th Segre class of the normal (resp. tangent) bundle of $\B$.
 Also, let $S\subset\PP^5$ denote  a smooth hyperplane section of $\B$. 
 Then, with this notation, we have:
\begin{enumerate}[(i)]
\item\label{cherniLemma} $c_1 = 2\lambda-2g+2$,
$c_2 = -29\lambda+16g+222$,
$c_3 = 230\lambda-102g-1788$;
\item\label{chernSEIILemma} $s_1 = -5\lambda-2g+2$,
$s_2 = -15\lambda+30g+208$,
$s_3 = 405\lambda-270g-3286$;
\item $\mathfrak{c}_1=5 \lambda+2 g-2$,
$\mathfrak{c}_2=-3 \lambda+12 g+100$,
$\mathfrak{c}_3=\lambda^2$;
\item $\mathfrak{s}_1 = -2\lambda+2g-2$,
$\mathfrak{s}_2 = -10\lambda-2g+114$,
$\mathfrak{s}_3 = \lambda^2+77\lambda-28g-756$;
\item\label{powerK} $K_{\B}\,H_{\B}^2 = -2\lambda+2g-2$,
$K_{\B}^2\,H_{\B} = -39\lambda+14g+336$,
$K_{\B}^3 = \lambda^2-77\lambda+14g+672$;
\item\label{sezHiperpianagener} $K_S\,H_S=-\lambda+2g-2$, $K_{S}^2 = -42\lambda+18g+332$,
      $c_2(\T_{S}) = -30\lambda+18g+220$.
\end{enumerate}
\end{lemma}
\begin{proof} 
 For every smooth threefold $\B$,
 by Hirzebruch-Riemann-Roch formula (see \cite[App. A, Exercise~6.7]{hartshorne-ag})
 and using that $s(\T_{\B})=c(\T_{\B})^{-1}$ we deduce:
\begin{align}
    s_1(\T_{\B})\,  H_{\B}^2 &= -c_1(\T_{\B})\,  H_{\B}^2=K_{\B}\,  H_{\B}^2 , \\
 s_2(\T_{\B})\,  H_{\B} &= c_1(\T_{\B})^2\,  H_{\B}-c_2(\T_{\B})\,  H_{\B} 
 = K_{\B}^2\,  H_{\B}-c_2(\T_{\B})\,  H_{\B}, \\
 \begin{split}
s_3(\T_{\B}) &= -c_1(\T_{\B})^3+2c_1(\T_{\B})\,  c_2(\T_{\B})-c_3(\T_{\B})  \\
&=  K_{\B}^3+48\chi(\O_{\B})-c_3(\T_{\B}),  
\end{split} \\
K_{\B}\,  H_{\B}^2 &= 2\,(g-1-H_{\B}^3), \\
 K_{\B}^2\,  H_{\B} &=  
12(\chi(\O_{\B}(H_{\B}))-\chi(\O_{\B})) -2\, H_{\B}^3 +3\, K_{\B}\,  H_{\B}^2  -c_2(\T_{\B})\,  H_{\B}.
\end{align}
Using that $\B$ is embedded in $\PP^6$, we also obtain the relation 
(see \cite[p.~543]{livorni-sommese-15}): 
\begin{equation}\label{doublePointFormula}
K_{\B}^3  
= c_3(\T_{\B})+7\,c_2(\T_{\B})\,  H_{\B}-48\,\chi(\O_{\B})+(H_{\B}^3)^2-35\,H_{\B}^3
 -21\,K_{\B}\,  H_{\B}^2-7\,K_{\B}^2\,  H_{\B}.
\end{equation}
Moreover, from the exact sequence
$0\rightarrow\mathcal{T}_{\B}\rightarrow 
\mathcal{T}_{\PP^6}|_{\B}\rightarrow\mathcal{N}_{\B,\PP^6}\rightarrow0$
and since $s(\N_{\B,\PP^6})=c(\N_{\B,\PP^6})^{-1}$ and $s(\T_{\B})=c(\T_{\B})^{-1}$, we deduce:
\begin{align}
c_1(\T_{\B})\,H_{\B}^2 &= 7\,H_{\B}^3+s_1(\N_{\B,\PP^6})\, H_{\B}^2,\\ 
c_2(\T_{\B})\,H_{\B} &= 21\,H_{\B}^3+7\,s_1(\N_{\B,\PP^6})\, H_{\B}^2+s_2(\N_{\B,\PP^6})\, H_{\B}, \\
 c_3(\T_{\B}) &= 35\,H_{\B}^3+21\,s_1(\N_{\B,\PP^6})\, H_{\B}^2+7\,s_2(\N_{\B,\PP^6})\, H_{\B}+s_3(\N_{\B,\PP^6}) ; 
\end{align}
\begin{align}
 c_1(\N_{\B,\PP^6})\,H_{\B}^2 &= 7\,H_{\B}^3+s_1(\T_{\B})\, H_{\B}^2  , \\
 c_2(\N_{\B,\PP^6})\,H_{\B} &=  21\,H_{\B}^3+7\,s_1(\T_{\B})\, H_{\B}^2+s_2(\T_{\B})\, H_{\B}  ,  \\
 c_3(\N_{\B,\PP^6}) &=    35\,H_{\B}^3+21\,s_1(\T_{\B})\, H_{\B}^2+7\,s_2(\T_{\B})\, H_{\B}+s_3(\T_{\B})   .
\end{align}
Now, using that $\B$ is the base locus of a Cremona transformation 
of $\PP^6$ of type $(3,5)$, 
from Fact~\ref{factSpecialCremona2} we obtain: 
\begin{align}
 1 = \deg_6(\varphi) &=
 -540\,H_{\B}^3-135 \, s_1(\N_{\B,\PP^6})\, H_{\B}^2-18\,s_2(\N_{\B,\PP^6})\, H_{\B}  
        -s_3(\N_{\B,\PP^6})+729 , 
\\
5 = \deg_5(\varphi) &=  -90\,H_{\B}^3-15\,s_1(\N_{\B,\PP^6})\, H_{\B}^2-s_2(\N_{\B,\PP^6})\, H_{\B}+243.
\end{align}
Finally,  for every smooth threefold $\B$, if $S$ is a smooth hyperplane section,
 from the exact sequence
$0\rightarrow \T_S \rightarrow \T_{\B}|_{S}\rightarrow \O_{S}(H_S)\rightarrow 0 $,
we get  
\begin{equation}\label{EquazioneC2TS}
c_2(\T_S)=c_2(\T_{\B})\,H_{\B}+K_S\,H_S
=c_2(\T_{\B})\,H_{\B}+K_{\B}\,H_{\B}^2+H_{\B}^3 ,
\end{equation} 
and we also have 
\begin{equation}\label{EquazioneKS2}
K_{S}^2=12\,\chi(\O_S)-c_2(\T_S)=
12\,(\chi(\O_{\B})-\chi(\O_{\B}(-H_{\B})))-c_2(\T_S)  .
\end{equation}
Now, using also Lemma \ref{hilbertPolB}, the proof is reduced to solving a system of linear equations with 
coefficients rational functions of $\lambda$ and $g$.
\end{proof}
\begin{lemma}\label{CremonaDegrees}
The projective degrees of $\varphi$ are as follows:
 $\deg_0(\varphi)=1$, $\deg_1(\varphi)=3$, $\deg_2(\varphi)=9$, 
 $\deg_3(\varphi)=27-\lambda$, $\deg_4(\varphi)=-7\lambda+2g+79$,
 $\deg_5(\varphi)=5$, $\deg_6(\varphi)=1$.
\end{lemma}
\begin{proof}
It follows directly from 
Lemma \ref{lemmaChern}(\ref{chernSEIILemma}) and Fact \ref{factSpecialCremona2}. 
\end{proof}
\subsection{Reducing to a short list of not excluded cases}
In this subsection we prove the following:
\begin{lemma}\label{lemma extreme}
 There are $33$ not excluded pairs $(\lambda,g)$. These are the following:
 $(8,1)$, $(9,3)$, $(9,4)$, $(10,5)$, $(10,6)$, $(11,7)$, $(11,8)$, $(11,9)$, 
 $(12,9)$, $(12,10)$, $(12,11)$, $(12,12)$, $(13,12)$,
$(13,13)$, $(13,14)$, $(13,15)$, $(14,14)$, $(14,15)$, 
$(14,16)$, $(14,17)$, $(14,18)$, $(15,17)$, $(15,18)$, $(15,19)$, $(15,20)$,
$(15,21)$, $(16,21)$, $(16,22)$, $(16,23)$, $(17,24)$, $(17,25)$, $(18,27)$, $(18,28)$.
 \end{lemma}
\begin{proof}
Firstly, we note that there are a finite number of not excluded pairs $(\lambda,g)$.
Indeed, since $\B$ is a smooth nondegenerate scheme 
 cut out by cubics and of 
  codimension $3$, we have $3\leq\lambda\leq27$. Further, we have $g\geq0$ and 
  the well-known 
  Castelnuovo’s bound (see e.g. \cite[p.~4]{besana-biancofiore-numerical}) provides
  an upper bound for $g$ as a function of $\lambda$. 
This rough argument leaves $889$ pairs $(\lambda,g)$.

Now, from \cite[Theorem~0.7.3]{livorni-sommese-15}, 
 for  any smooth threefold $(\B,H_{\B})$, denoting by $S$ a smooth hyperplane section, 
 one has the inequalities:
 \begin{gather}
   K_{\B}^3+6\,K_{\B}^2\,H_{\B}+15\,K_{\B}\,H_{\B}^2+20\,H_{\B}^3-c_3(\T_{\B})+
         48\,\chi(\O_{\B})-6\,c_2(\T_{\B})\,H_{\B}\geq 0, \label{LiSoIneq1} \\
   K_S^2+4\,K_S\,H_S+6\,H_S^2\geq c_2(\T_S), \label{LiSoIneq2}\\
   2\,c_2(\T_S)\geq c_3(\T_{\B})-2g+2, \label{LiSoIneq3}\\
   -24\,\chi(\O_{\B})+3\,K_{\B}^2\,H_{\B} 
         +15\,K_{\B}\,H_{\B}^2+2\,c_2(\T_{\B})\,H_{\B}+20\,H_{\B}^3+c_3(\T_{\B})\geq0 . \label{LiSoIneq4}
 \end{gather}
 So, applying Lemmas \ref{hilbertPolB} and \ref{lemmaChern}, respectively, we obtain:
\begin{equation}
 \begin{array}{ll} 
 g \leq (\lambda^2+7\lambda-102)/10, &
 g \geq (5\lambda-52)/4, \\
 g \geq (145\lambda-1113)/70, &
 g \geq (147\lambda-1242)/74. 
 \end{array}
 \end{equation}
In particular, the first inequality provides a refinement of the Castelnuovo’s bound. 
All these inequalities restrict the set of not excluded pairs $(\lambda,g)$ 
to a set of $312$ elements.

Finally, using Lemma \ref{CremonaDegrees}, 
we apply all the inequalities in Fact \ref{factCremona}.
This turns out to be equivalent to applying just the following: 
\begin{equation}\label{essentialCremonaIneq}
 \begin{array}{c} 
 \begin{array}{cc}
\deg_1(\varphi)\,\deg_3(\varphi)\geq\deg_4(\varphi), &
\deg_1(\varphi)\,\deg_4(\varphi)\geq\deg_5(\varphi), 
\end{array} \\
\begin{array}{cc}
{\deg_3(\varphi)}^2\geq\deg_2(\varphi)\,\deg_4(\varphi), &
{\deg_4(\varphi)}^2\geq\deg_3(\varphi)\,\deg_5(\varphi), 
\end{array} \\
\begin{array}{c}
{\deg_5(\varphi)}^2\geq\deg_4(\varphi)\,\deg_6(\varphi).
 \end{array}
 \end{array}
 \end{equation}
Which respectively become:
\begin{equation}
 \begin{array}{c} 
 \begin{array}{cc}
4 \lambda-2 g+2\geq0, &
-21 \lambda+6 g+232\geq0,
\end{array} \\
\begin{array}{cc}
\lambda^2+9 \lambda-18 g+18\geq0, &
49 \lambda^2-28 \lambda g+4 g^2-1101 \lambda+316 g+6106\geq0,
\end{array} \\
\begin{array}{c}
7 \lambda-2 g-54\geq0.
 \end{array}
 \end{array}
 \end{equation}
This restricts the set of pairs $(\lambda,g)$ to the set of $33$ elements as stated.
\end{proof}
\subsection{Applying  adjunction theoretic results}
The following result, thanks to Fact \ref{listaCasi}, say us that $(\B,H_{\B})$ admits a unique minimal reduction.
\begin{lemma}\label{LemmaThereExistsReduction}
 The adjunction map $\Phi=\Phi_{|K_{\B}+2H_{\B}|}$ 
 is defined and has image of dimension $3$.
 In other words, $\B\subset\PP^6$ is not a del Pezzo variety, 
 and it is different from a scroll over either a curve or a surface, 
 as well as from a quadric fibration over a curve.
\end{lemma}
\begin{proof}
If $\B$ is a scroll over a curve, we must have $(K_{\B}+3\,H_{\B})^3=0$;
while if $\B$ is a scroll over a surface, 
or a quadric fibration over a curve, or a del Pezzo variety, then 
we must have $(K_{\B}+2\,H_{\B})^3=0$.
From Lemma \ref{lemmaChern}(\ref{powerK}), 
these two relations, respectively, become  
$ \lambda^2-455\lambda+194g+3642 =0$ and $ \lambda^2-327\lambda+122g+2664=0$. 
But there is 
 no pair $(\lambda,g)$, not excluded by Lemma \ref{lemma extreme},
 that satisfies any of them.
\end{proof}
Let $(\R,H_{\R})$ denote the minimal reduction of $(\B,H_{\B})$ 
and let $\nu=\nu(\B,H_{\B})$ be the number of exceptional divisors on $(\B,H_{\B})$. 
In the following result, we compute the numerical invariants of $\R$ as 
functions of $\lambda$, $g$, and $\nu$.
\begin{lemma}\label{invariantiRid} We have
 \begin{enumerate}[(i)] 
 \item\label{hilberPolynomialR} $\chi(\O_{\R}(tH_{\R}))
=  (\lambda+\nu)\,\binom{t+3}{3} +(-\lambda-g-\nu+1)\,\binom{t+2}{2} +(-6\lambda+4g+45)\,(t+1)+(14\lambda-6g-113) $;
\item\label{powerKR} 
 $K_{\R}\,H_{\R}^2 = -2d+2g-2\nu-2 $, 
 $K_{\R}^2\,H_{\R} = -39\lambda+14g+4\nu+336$, 
 $K_{\R}^3 = \lambda^2-77\lambda+14g-8\nu+672$; 
\item 
 $c_1(\T_{\R})\,H_{\R}^2 = 2\lambda-2g+2\nu+2$, 
 $c_2(\T_{\R})\,H_{\R} = -29\lambda+16g+222$, 
 $c_3(\T_{\R})=230\lambda-102g-2\nu-1788 $; 
\item\label{seziperrid} for a smooth element $\mathfrak{P}\in|H_{\mathfrak{R}}|$, 
 $K_{\mathfrak{P}}^2 = -42\lambda+18g+\nu+332$,
      $c_2(\T_{\mathfrak{P}}) = -30\lambda+18g-\nu+220$. 
\end{enumerate}
\end{lemma}
\begin{proof}
Let $\pi:\B\rightarrow\R$ be the map of the blowing-up along 
the points $p_1,\ldots,p_{\nu}$,
and denote by $E_i=\pi^{-1}(p_i)$ the exceptional divisors.
Since $K_{\B}\approx \pi^{\ast}K_{\R}+2\sum_{i=1}^{\nu}E_i$ 
and $H_{\B}\approx\pi^{\ast}H_{\R}-\sum_{i=1}^{\nu}E_i$, one obtains 
$K_{\R}\,H_{\R}^2=K_{\B}\,H_{\B}^2-2\nu$, 
$K_{\R}^2\,H_{\R}=K_{\B}^2\,H_{\B}+4\nu$, and
$K_{\R}^3=K_{\B}^3-8\nu$. One also has
$c_1(\T_{\R})\,H_{\R}^2=c_1(\T_{\B})\,H_{\B}^2+2\nu$, and 
from \cite[p.~609]{griffiths-harris} it follows that 
$c_2(\T_{\R})\,H_{\R}=c_2(\T_{\B})\,H_{\B}$.
Now,  using Hirzebruch-Riemann-Roch formula,   
one deduces that 
for every $t\in\ZZ$, $\chi(\O_{\R}(tH_{\R}))=\chi(\O_{\B}(tH_{\B}))+\nu\,(t^3+3t^2+2t)/6$;
in particular, $\chi(\O_{\R}(tH_{\R}))=\chi(\O_{\B}(tH_{\B}))$ for $t=0,-1,-2$. 
As an application of double point formulas, 
see \cite[Lemma~3.3]{besana-biancofiore-numerical}, one deduces that 
$
s_3(\T_{\R}) = -21\,s_1(\T_{\R})\,H_{\R}^2-7\,s_2(\T_{\R})\,H_{\R}  +(H_{\R}^3-\nu)^2 -35\,H_{\R}^3+15\nu 
$,
and from this and \eqref{doublePointFormula}
one has:
\begin{align*}
 c_3(\T_{\R}) &= K_{\R}^3+48\,\chi(\O_{\R})-s_3(\T_{\R}) \\
  &= K_{\R}^3+48\,\chi(\O_{\R})+21\,K_{\R}\,H_{\R}^2+7\,(K_{\R}^2\,H_{\R}-c_2(\T_{\R})\,H_{\R}) \\ & \qquad  -(H_{\R}^3-\nu)^2 +35\,H_{\R}^3-15\,\nu \\ 
  &= 
  -7\,c_2(\T_{\B})\,H_{\B}
   +48\,\chi(\O_{\B})
  -\lambda^2  +35\lambda
  +21\,K_{\B}\,H_{\B}^2
 \\ & \qquad  +7\,K_{\B}^2\,H_{\B}
  +K_{\B}^3
  -2\nu  \\
 \mbox{(by \eqref{doublePointFormula})} 
 &= c_3(\T_{\B})-2\nu.
  \end{align*}
Finally, part (\ref{seziperrid}) is obtained 
through the analogue formulas of  
(\ref{EquazioneC2TS}) and (\ref{EquazioneKS2}).
\end{proof} 
\begin{proposition}\label{casiPossi}
One of the two following possibilities occurs:
\begin{enumerate}[(A)]
 \item $(\R,H_{\R})$ is of log-general type, and $(\lambda,g,\nu)\in\{(14, 15, 0),\  (18, 27, 0)\}$;
 \item $(\R,H_{\R})$ is a conic bundle over a surface, and 
 $\nu=\lambda^2-199\lambda+62g+1674$.
\end{enumerate} 
\end{proposition}
\begin{proof}
We apply Fact \ref{NefBigFact} in order to  recognize when $K_{\R}+H_{\R}$ is  nef and big.
One sees easily that $(\R,H_{\R})$ cannot be as in Fact \ref{NefBigFact}, cases (\ref{P3red}) and (\ref{Q3red}).
 Assume that $(\R,H_{\R})$ is as in Fact \ref{NefBigFact}, case (\ref{VeroneseFibration}).
 Then, from Lemma \ref{invariantiRid}(\ref{powerKR}), we obtain the two conditions:
 \begin{align}
  0 &= (2K_{\R}+3H_{\R})^3     
    = 8\lambda^2-2101\lambda+724g-\nu+17364 , \\
 0 &=  (2K_{\R}+3H_{\R})^2\,H_{\R}     
    = -171\lambda+80g+\nu+1320  .
 \end{align}
Thus we must have 
$
\nu=
8 \lambda^2-2101 \lambda+724 g+17364 = 171 \lambda-80 g-1320 
$,
but 
there are no pairs $(\lambda,g)$,
not excluded by Lemma \ref{lemma extreme}, that satisfy it. 
Similarly, if $(\R,H_{\R})$ is as in Fact \ref{NefBigFact}, case 
(\ref{delPezzoFibration}),  we deduce that
$
\nu=\lambda^2-199 \lambda+62 g+1674 = 42 \lambda-18 g-332 
$,
which is again impossible.
Now, assume that $(\R,H_{\R})$ is as in Fact \ref{NefBigFact}, case (\ref{Mukaired}).
The three conditions $K_{\R}^3+H_{\R}^3=K_{\R}^2\,H_{\R}-H_{\R}^3=K_{\R}\,H_{\R}^2+H_{\R}^3=0$
become
$
\lambda^2-76 \lambda+14 g-7 \nu+672 = -40 \lambda+14 g+3 \nu+336 = -\lambda+2 g-\nu-2 =0
$,
but there is no $3$-tuple of integers that satisfies them.
Finally, if $(\R,H_{\R})$ is as in Fact \ref{NefBigFact}, case (\ref{quadricFibrOverSurface}),
i.e. a conic bundle over a surface,
we then deduce  the relation:
$
  0 = (K_{\R}+H_{\R})^3  
  = \lambda^2-199\lambda+62g-\nu+1674 
  $. 

  Now, if  $K_{\R}+H_{\R}$ is nef and big, 
  using again Lemma \ref{invariantiRid}, 
we can apply all the inequalities contained in Fact \ref{LogGeneralIneq}. 
For example, the obvious inequality $d_1\geq 1$ translates to 
$\nu\leq -\lambda+2g-3$, and in particular we see that
only finitely many triples $(\lambda,g,\nu)$ 
are not excluded (more precisely, there are $478$ not excluded triples).
If now one applies the inequality $d_1^2-d_2\,d_0\geq 0$, 
which translates to 
$43 \lambda^2-22 \lambda g+4 g^2+43 \lambda \nu-22 g \nu-328 \lambda-8 g-328 \nu+4\geq 0$,
then only  $18$ triples are not excluded. Among these, 
only $(14, 15, 0)$ and  $(18, 27, 0)$
are not excluded by $d_2=-42 \lambda+18 g+\nu+332\geq 1$ 
and $d_3=\lambda^2-199 \lambda+62 g-\nu+1674\geq 1$ (and not excluded 
by any other inequality stated 
in Fact \ref{LogGeneralIneq}).
\end{proof}
\subsection{Applying  Le Barz's formula}
We can apply Fact \ref{multisecants} to the general hyperplane section $S\subset\PP^5$ 
of $\B\subset\PP^6$, by taking into account the following:
\begin{itemize}
\item $S\subset\PP^5$ can contain at most finitely many lines, 
and all them have self-intersection $\leq -1$;
\item the exceptional divisors of $\B$ correspond to the lines 
$l\subset S$ with self-intersection $l^2=-1$; and so we have 
\begin{equation}\label{autointersezione}
\sum_{\begin{subarray}{l} l\, \mathrm{line}\\l\subset S\end{subarray}}\binom{5+l^2}{4} = 
\nu + \sum_{\begin{subarray}{l} l\, \mathrm{line},\,l\subset S\\ l^2\leq -6\end{subarray}}\binom{5+l^2}{4}
\geq \nu ;
\end{equation}
\item since $S\subset\PP^5$ is cut out by cubics, it cannot have  $4$-secant lines.
\end{itemize}
Thus, using (\ref{autointersezione}) and Lemma \ref{lemmaChern}(\ref{sezHiperpianagener}), 
from Fact \ref{multisecants} we deduce the following inequality:
\begin{equation}\label{patrick4Ineq}
 \begin{split}
\nu &\leq (1/8) {\lambda}^{4}+(3/4) {\lambda}^{3}-3 {\lambda}^{2} g-(453/8) {\lambda}^{2}+20 {\lambda} g   +13
     g^{2}   \\ & \qquad +(2835/4) {\lambda}-73 g-2894.
 \end{split}
     \end{equation}
Now, from Proposition \ref{casiPossi} and the inequality (\ref{patrick4Ineq}), one deduces  
 that  there are only the two  following possibilities:
 \begin{enumerate}[(A')]
 \item\label{case1Log} $(\lambda,g,\nu)=(14,15,0)$ and $(\B,H_{\B})=(\R,H_{\R})$ is of log-general type;
 \item\label{case2Fibration} $(\lambda,g,\nu)=(13,12,0)$ and  $(\B,H_{\B})=(\R,H_{\R})$ has the structure 
 of a conic bundle  over a polarized surface $(Y,H_{Y})$.
\end{enumerate} 
\subsection{Conclusion of the proof} 
In case (\ref{case2Fibration}') above, (by Lemma  \ref{invariantiRid}) 
we get 
$h^0(K_{\R}+H_{\R})=h^3(-H_{\R})=-\chi(\O_{\R}(-H_{\R}))=3$
and $d_2(\R,H_{\R})=2$. From this, it follows that $(Y,H_{Y})=(\PP^2,\O_{\PP^2}(1))$, 
see \cite[Proposition~11.5]{besana-biancofiore-numerical}.
Assume now that $\B$ is as in case (\ref{case1Log}'). 
 We first show that $K=K_{\B}$ is numerically trivial 
with the same argument as in \cite[(3.8), p.~69]{debarre}.
So, let $C$ be a fixed irreducible curve on $\B$. A standard argument 
(see for example {\it loc. cit.}) and Bertini  theorem assures that 
for   $m$ sufficiently large there exists an irreducible and reduced $Y\in |mH|$ containing  $C$ 
($H=H_{\B}$ stands for the hyperplane divisor).
From Lemma \ref{lemmaChern}(\ref{powerK}),
we see that  $K|_{Y}\cdot H|_{Y}=K\cdot H\cdot (mH)=m(K\cdot H^2)=0$ and that 
$K|_{Y}^2=K\cdot K\cdot (mH)=m(K^2\cdot H)=0$.  Thus, applying  Hodge index theorem  
on the irreducible surface $Y$ (see \cite[Th\'eor\`eme 7.1]{kleiman}), we deduce that $K|_{Y}$ is numerically trivial.
 It follows that $K\cdot C=K|_{Y}\cdot C=0$, and hence that $K$ is numerically trivial.
 Now, from \cite[Theorem~8.2]{kawamata1985}, we deduce that 
 the Kodaira dimension of $\B$ is zero, and consequently,
 from \cite[Corollary~1.1.2]{livorni-sommese-15}, we deduce that $K$ is trivial. 
 Now we conclude that $\B$ is Pfaffian from the main result in \cite{walter-pfaffian}.
 \section{Proof of Theorem \ref{CremonaP7}}\label{sec cubo-cubic}
 Let $\varphi$ be a special Cremona transformation of $\PP^7$.
 Denote 
  by $\B$ (resp. $\B'$) the base locus of $\varphi$ (resp. $\varphi^{-1}$),
  and by $(\delta_1,\delta_2)$ the type of $\varphi$. 
  From \eqref{dimensionBaseLocus}
and 
since one has  $\dim\B\leq4$, it follows $\dim\B=\dim\B'=4$ and $\delta_1=\delta_2=3$.
Thus the restriction of $\varphi$ to a general hyperplane $\PP^6\subset\PP^7$ 
is a cubo-cubic birational transformation of $\PP^6$ into a cubic hypersurface of $\PP^7$
whose base locus is a smooth irreducible threefold $X\subset\PP^6$, which is a
general hyperplane section of $\B$. 
We now conclude the proof of Theorem \ref{CremonaP7} 
by adapting the proof of Theorem \ref{MainTheorem}
in order to get information on $X$
and hence on $\B$.
\begin{proof}[Proof of Theorem \ref{CremonaP7}]
Keep the notation as above.
As $\codim\B=\codim\B'\geq 3$
we  have $\deg_{i}(\varphi)=\deg_{7-i}(\varphi)=3^i$ for $i=0,1,2$, and 
by \eqref{segresDegs} 
we can express 
$\deg_{3}(\varphi)$ and $\deg_{4}(\varphi)$ 
as functions of $\lambda=\deg\B$ 
and of the sectional genus $g$ of $\B$; namely, we have: 
\begin{equation}\label{multidegreeCremonaP7}
\begin{array}{lll}
\deg_1(\varphi)=3, &\quad \deg_2(\varphi)= 9, &\quad \deg_3(\varphi)= -\lambda+27, \\ \deg_4(\varphi)= -7 \lambda+2 g+79, &\quad  \deg_5(\varphi)= 9, &\quad \deg_6(\varphi) =3 .
\end{array}
\end{equation}
By \eqref{vanishingLemma}
 we deduce that $\B$ is projectively normal 
and its Hilbert polynomial, $P_{\B}(t)$, satisfies the conditions:
$P_{\B}(3)=112$, $P_{\B}(2)=36$, and $P_{\B}(1)=8$, from which 
it follows that $P_{\B}(t)$ has the following expression: 
\begin{multline*}
\lambda\,\binom{t+4}{4}+(-\lambda-g+1)\,\binom{t+3}{3}+(-6 \lambda+4 g+44)\,\binom{t+2}{2} \\ +(14 \lambda-6 g-110)\, (t+1)+(-11 \lambda+4 g+92)  .
\end{multline*}
From \eqref{segresDegs} and \eqref{multidegreeCremonaP7} 
one determines 
the degrees of the Segre 
classes of the normal bundle of $\B$, and then, 
from the formulas in the proof of Lemma~\ref{lemmaChern}, 
  one obtains the numerical invariants of $X$
as functions of $\lambda$ and $g$.
In particular, 
denoting by 
$K_{X}$ and $H_{X}$, respectively, a canonical divisor and
 a hyperplane section divisor of $X$, and by $S$ a smooth member of $|H_X|$,
one has:
\begin{gather}\label{powersKCuboCubic}
\begin{split}
 K_{X}\,H_{X}^2 &=  -2\lambda +2g-2,\quad   \\   
K_{X}^2\,H_{X} &= -39\lambda+14g+328, \quad  \\     
K_{X}^3 &=  \lambda^2-77\lambda+14g+646 ; 
\end{split}\\
\label{KSCubo}
K_{S}^2 = -42 \lambda+18 g+324, \quad 
      c_2(\T_{S}) =  -30 \lambda+18 g+216 .
\end{gather}
Now we proceed as in Lemma \ref{lemma extreme}, by applying 
the inequalities 
\eqref{LiSoIneq1}, \eqref{LiSoIneq2}, \eqref{LiSoIneq3} and \eqref{LiSoIneq4} 
to the threefold $X$, and the inequalities \eqref{essentialCremonaIneq} 
to the Cremona transformation $\varphi$.
Then, 
using also  
that $\lambda$ and $g$ are
non-negative integers bounded above, respectively, by 
$27$ and by the Castelnuovo’s bound, 
one can show 
that there are only $27$ not excluded pairs $(\lambda,g)$, which are the 
following:
$(8,2 )$,  $(9,4 )$,  $(10,6 )$,  $(10,7 )$,  $(11,8 )$,  $(11,9 )$,  $(11,10 )$,  $(12,10 )$,  $(12,11 )$,
      $(12,12 )$,  $(12,13 )$,  $(13,12 )$,  $(13,13 )$,  $(13,14 )$,  $(13,15 )$,  $(13,16 )$,  $(14,15 )$,
      $(14,16 )$,  $(14,17 )$,  $(14,18 )$,  $(15,19 )$,  $(15,20 )$,  $(15,21 )$,  $(16,22 )$,  $(16,23 )$,
      $(17,25 )$,  $(18,28 )$.
      One then sees, using \eqref{powersKCuboCubic} and 
the same argument as in the proof of Lemma \ref{LemmaThereExistsReduction}, that 
 $(X,H_X)$
admits a unique minimal reduction $(X',H')$,
and we denote by $\nu$ the number of exceptional divisors on 
$(X,H_{X})$. 
Now, quite similarly as we have deduced the inequality \eqref{patrick4Ineq}, 
from  \eqref{KSCubo} and Fact \ref{multisecants}, 
we  deduce the following:
\begin{equation}\label{multisecantCuboCubic}
\begin{split}
 \nu&\leq (1/8)  \lambda^4+(3/4)  \lambda^3-3  \lambda^2 g-(445/8)  \lambda^2+20  \lambda g+13 g^2
 \\ & \qquad +(2815/4)  \lambda-83 g-2873 .
\end{split}
\end{equation}
In particular, we see that
only finitely many triples $(\lambda,g,\nu)$ 
are not excluded (precisely, there are $859$ not excluded triples).
We now show that $(X',H')$ is not of log-general type.
Assume by contradiction that $K'+H'$ is nef and big, where 
$K'$ stands for the canonical divisor of $X'$.
From \eqref{powersKCuboCubic} (see also the proof of Lemma \ref{invariantiRid}) 
we can determine the pluridegrees $d_0,\ldots,d_3$ of $(X',H')$ as 
functions of $\lambda$, $g$, and $\nu$:
\begin{equation}\label{pluridegreesCuboCubic}
\begin{split}
 d_0=\lambda+\nu, &\qquad 
 d_1=-\lambda+2 g-\nu-2,\\  
 d_2=-42 \lambda+18 g+\nu+324, &\qquad 
 d_3=\lambda^2-199 \lambda+62 g-\nu+1624.
\end{split}
\end{equation}
Then by Fact \ref{LogGeneralIneq} 
in particular we  deduce:
\begin{gather*}
 d_1\geq 1,\quad d_2\geq 1, \quad d_3\geq 1,  \\
 d_1^2-d_2 d_0 = 43 \lambda^2-22 \lambda g+4 g^2+43 \lambda \nu-22 g \nu-320 \lambda-8 g-320 \nu+4 \geq0.
\end{gather*}
But one can sees that 
the set of triples $(\lambda,g,\nu)$ satisfying these inequalities, 
with $(\lambda,g)$ being one of the $27$ not excluded pairs,
consists of only one element which is $(18,28,0)$.
So  \eqref{pluridegreesCuboCubic} becomes: 
\begin{equation*} 
 d_0= 18,\quad 
 d_1= 36,\quad 
 d_2= 72,\quad 
 d_3= 102.
\end{equation*}
This is impossible by \cite[Lemma~1.1, (1.1.2)]{beltrametti-biancofiore-sommese-loggeneral},
because the equality 
$d_1^2-d_2 d_0=0$ should  imply $d_2^2-d_3 d_1=0$,
and this is not the case.
It follows that $(X',H')$ 
is as one of the six cases of Fact \ref{NefBigFact}.
Cases \eqref{P3red}, \eqref{Q3red}, \eqref{VeroneseFibration}, and \eqref{Mukaired} 
 are easily excluded,  
arguing as at the beginning of the proof of  Proposition~\ref{casiPossi}.
Therefore 
$(X',H')$ is as in case either \eqref{delPezzoFibration} or \eqref{quadricFibrOverSurface}.
 Then we must have $d_3=(K'+H')^3=0$, 
which by \eqref{pluridegreesCuboCubic} means  
$\nu =  \lambda^2-199 \lambda+62 g+1624 $.
But from this and \eqref{multisecantCuboCubic} we get 
that only the triple $(\lambda,g,\nu)=(12,10,0)$ is not excluded, and in particular
 \eqref{pluridegreesCuboCubic} becomes: 
$d_0= 12$, 
 $d_1= 6$, 
 $d_2= 0$, 
 $d_3= 0$.  
By $d_2=0$ we deduce that case \eqref{quadricFibrOverSurface} is impossible. 
Thus we conclude that 
 $(X',H')$ is as in case \eqref{delPezzoFibration},
i.e. a del Pezzo fibration over a  polarized curve $(C,H_{C})$, and 
$(X,H_{X})=(X',H')$ is a minimal threefold of degree $12$ and sectional genus $10$.
Moreover,  denoting by 
$F$ the generic fibre of the fibration,
the following facts are known (see \cite[p.~13]{besana-biancofiore-numerical}): 
\begin{equation*}
 g(C)=1-\chi(\O_X) ,  \quad  \deg(H_C) = -\chi(\O_X)-\chi(\O_X(-H_X)) ,\quad  \deg(F)=d_1/\deg(H_C), 
\end{equation*}
from which it follows that $(C,H_C)=(\PP^1,\O_{\PP^1}(1))$ and $\deg(F)=6$. 
\end{proof} 
\section{Examples}\label{sez examples}
\subsection{A threefold of degree \texorpdfstring{$14$}{14} and sectional genus \texorpdfstring{$15$}{15}}\label{degree14sec}  
As it was pointed out in 
\cite[Subsection~4.2]{ein-shepherdbarron} (see also
\cite[Example on p.~282]{crauder-katz-1991}), 
the Pfaffians of the principal $6\times6$ minors of a general $7\times7$ 
skew-symmetric matrix of linear forms on $\PP^6$ give a Cremona 
transformation of $\PP^6$. Its base locus $X\subset\PP^6$, i.e. 
the vanishing locus of the Pfaffians, is a smooth and irreducible threefold 
of degree $14$ and sectional genus $15$, with trivial canonical bundle 
and $h^1(\O_X)=0$. Therefore, examples for case (\ref{LogGeneralCaseThm}) of Theorem \ref{MainTheorem} exist.
One can also find explicit equations for such an $X$ 
with the help of a computer algebra system as {\sc Macaulay2}.
\subsection{A threefold of degree \texorpdfstring{$13$}{13} and sectional genus \texorpdfstring{$12$}{12}}
Here, we construct an example of special Cremona transformation 
as in case (\ref{FibrationCaseThm}) of Theorem \ref{MainTheorem}.
We first construct a threefold in $\PP^6$ of degree $13$ and sectional genus $12$, 
 but which turns out to be singular.
 Let $C\subset\PP^3$ be the isomorphic projection of the 
quintic rational normal curve together with an embedded point $p_0$, and 
let $p_1,\ldots,p_5$ be $5$ general points on a general plane passing through $p_0$.
Then the homogeneous ideal 
of the non-reduced scheme $C\cup\{p_1,\ldots,p_5\}\subset\PP^3$ is 
generated by $7$ quartics,
which define a quarto-quadric birational map $\PP^3\dashrightarrow\PP^6$. The image of this map is
an irreducible singular threefold of degree $13$, sectional genus $12$, and cut out by $7$ cubics. 
These $7$ cubics give a cubo-quintic Cremona transformation of $\PP^6$.
The next step is to deform this Cremona transformation 
in order to obtain a special Cremona transformation with the same invariants.
The idea is to take random restrictions and liftings of the map, but 
this is easier to do than to explain. In the following {\sc Macaulay2} code, 
we get the desired example.
{ \footnotesize
\begin{verbatim}
Macaulay2, version 1.9.2.1
with packages: ConwayPolynomials, Elimination, IntegralClosure, 
               LLLBases, PrimaryDecomposition, ReesAlgebra, TangentCone
i1 : loadPackage "Cremona";
i2 : K = QQ; P1 = K[u_0,u_1]; P3 = K[t_0..t_3]; P6 = K[x_0..x_6];
i6 : curve = trim kernel map(P1,P3,{u_1^5,u_0^5-u_0^3*u_1^2,u_0^4*u_1,u_0*u_1^4});
o6 : Ideal of P3
i7 : points = intersect {ideal(t_3,t_2,t_1),ideal(t_3,t_2,t_0-t_1),
              ideal(t_2-t_3,t_1+t_3,t_0-t_3),ideal(t_2-t_3,t_1+t_3,t_0+t_3),
              ideal(t_2-t_3,t_1-t_3,t_0),ideal(t_2-t_3,t_1-t_3,t_0+t_3)};
o7 : Ideal of P3
i8 : time phi = rationalMap kernel(map(P3,P6,gens saturate(points * curve)),3)
     -- used 1.77549 seconds
o8 = cubic rational map from PP^6 to PP^6
o8 : RationalMap
i9 : time phi = lift invertBirMap rationalMap(flatten(
                       (invertBirMap phi)|ideal(x_0-x_1,x_2-x_4)),Dominant=>3)
     -- used 104.747 seconds
o9 = cubic rational map from PP^6 to PP^6
o9 : RationalMap
i10 : time phi = invertBirMap lift invertBirMap rationalMap(
                       flatten(phi|ideal(x_0)),Dominant=>true) 
     -- used 40.6531 seconds
o10 = cubic rational map from PP^6 to PP^6
\end{verbatim}
} \noindent 
Furthermore, the package \href{https://github.com/Macaulay2/M2/blob/master/M2/Macaulay2/packages/Cremona.m2}{\emph{Cremona}} (v$\geq$3.9.1) provides the method \texttt{specialCremonaTransformation} 
 which is able to produce examples of special Cremona transformations according to Table \ref{table: classCremona}.
 In particular, we can obtain the example above as follows:
{ \footnotesize
\begin{verbatim}
i11 : time phi = specialCremonaTransformation 12
     -- used 0.208033 seconds
o11 = Cremona transformation of PP^6 of type (3,5)
o11 : RationalMap
i12 : time projectiveDegrees phi
     -- used 0.000019401 seconds
o12 = {1, 3, 9, 14, 12, 5, 1}
o12 : List
\end{verbatim}
} \noindent 
Once we have the explicit transformation, 
the unique non-trivial condition to check is the smoothness of the base locus. This can be done 
 by reducing the characteristic (see \cite[Section~2.5]{note3}).
 \begin{remark}
Following \cite[Example~5.4 and Theorem~4.1]{migliore-peterson}, 
one can construct another smooth threefold
  in $\PP^6$ of degree $13$ and sectional genus $12$,
  but which is not cut out by cubics.
  \end{remark}
 \subsection{A threefold of degree \texorpdfstring{$12$}{12} and sectional genus \texorpdfstring{$10$}{10}}
According to \cite[Case~6a, p.~260]{threefoldsdegree12},
a threefold in $\PP^6$ of degree $12$, sectional genus $10$, 
with a structure of fibration over $\PP^1$, the fibers of which 
are sextic del Pezzo surfaces, exists. 
This threefold is contained in 
 a quadric hypersurface of $\PP^6$, 
so that it 
does not define a cubo-cubic transformation 
as 
the one constructed via 
 restriction to a general $\PP^6$ 
of the transformation in Theorem~\ref{CremonaP7}.\footnote{Note 
that there also exist
smooth curves of degree $6$ and genus $3$ in $\PP^3$, contained in
a quadric surface, 
that do not define 
cubo-cubic transformations as in
 case \eqref{cubocubicinP4} of Theorem \ref{classificazioneCasoCurve}.}

Moreover, 
a cubo-cubic Cremona transformation  of $\PP^7$ 
(linearly equivalent to an involution) 
whose base locus is  
a fourfold of degree $12$ and sectional genus $10$ exists.
This is the polar map
of the dual hypersurface of $\PP^1\times\PP^1\times\PP^1\subset\PP^7$; 
its base locus consists of the union of three $\PP^1\times \PP^3\subset\PP^7$ 
intersecting pairwise in a $\PP^1\times\PP^1\times\PP^1\subset\PP^7$.

\subsection{Two threefolds of degree \texorpdfstring{$10$}{10} and sectional genus \texorpdfstring{$6$}{6}}
We conclude with two further examples of cubic birational transformations of $\PP^6$.

\subsubsection{} The $3\times 3$ minors of a general $3\times5$ 
 matrix of linear forms on $\PP^6$ give a special cubo-cubic birational 
transformation of $\PP^6$ into the Grassmannian $\mathbb{G}(2,4)\subset\PP^9$. 
Its base locus $X\subset\PP^6$, i.e. 
the vanishing locus of the $3\times 3$ minors, is a scroll over $\PP^2$ of 
degree $10$ and sectional genus~$6$.
 
\subsubsection{} The special quadro-cubic birational transformation of $\PP^8$ into a factorial cubic 
hypersurface of $\PP^9$, described in \cite[Example~5.5]{note},
can be extended (using the same method described in \cite{note3}) 
to a  quadro-cubic Cremona transformation of $\PP^9$.
The restriction to a general $\PP^6\subset\PP^9$ of the inverse of 
this Cremona transformation gives us an example of 
special cubo-quadric birational transformation of $\PP^6$ into a factorial 
complete intersection of three quadrics in $\PP^9$. Its base locus is a smooth irreducible 
threefold of degree $10$ and sectional genus $6$ that has a structure of scroll over 
$\PP^2$ with four double points blown up; the existence of such a threefold had been left undecided
 in \cite{fania-livorni-ten}.

\subsection*{Acknowledgements}
 I wish to thank Francesco Russo 
 for valuable communications and
 for  posing  to me the 
problem studied in this paper. 
  
\bibliographystyle{amsalpha}
\bibliography{bibliography}

\end{document}